\RequirePackage[l2tabu, orthodox]{nag} 
 
\documentclass{amsart}

\pdfoutput=1

\usepackage[utf8]{inputenc} 
\usepackage[colorlinks, linkcolor=blue,citecolor=blue, breaklinks=true]{hyperref} 
\usepackage[stretch=10, shrink=10]{microtype} 
\usepackage{amscd, amssymb}  
\usepackage{commath} 
\usepackage{mathtools} 
\usepackage[initials,msc-links,nobysame]{amsrefs}[2007/10/22]

\usepackage{paralist,enumitem}
\setlist{listparindent=0pt,parsep=3pt}

\newcommand{\TitleWithUrl}[1]{\IfEmptyBibField{doi}%
  {\IfEmptyBibField{url}{\textit{#1}}%
    {\IfEmptyBibField{eprint}{\href {\BibField{url}}{\textit{#1}}}{\textit{#1}}}%
    }%
  {\href {https://doi.org/\BibField{doi}}{\textit{#1}}}}
\renewcommand{\eprint}[1]{\IfEmptyBibField{url}{\url{#1}}%
  {\href {\BibField{url}}{#1}}}

\BibSpec{article}{%
    +{}  {\PrintAuthors}                {author}
    +{,} { \TitleWithUrl}               {title}
    +{.} { }                            {part}
    +{:} { \textit}                     {subtitle}
    +{,} { \PrintContributions}         {contribution}
    +{.} { \PrintPartials}              {partial}
    +{,} { }                            {journal}
    +{}  { \textbf}                     {volume}
    +{}  { \PrintDate}                {date}
    +{,} { \issuetext}                  {number}
    +{,} { \eprintpages}                {pages}
    +{,} { }                            {status}
    +{,} { available at arXiv:\eprint}        {eprint}
    +{}  { \parenthesize}               {language}
    +{}  { \PrintTranslation}           {translation}
    +{;} { \PrintReprint}               {reprint}
    +{.} { }                            {note}
    +{.} {}                             {transition}
    +{}  {\SentenceSpace \PrintReviews} {review}
}

\BibSpec{collection.article}{%
    +{}  {\PrintAuthors}                {author}
    +{,} { \TitleWithUrl}                     {title}
    +{.} { }                            {part}
    +{:} { \textit}                     {subtitle}
    +{,} { \PrintContributions}         {contribution}
    +{,} { \PrintConference}            {conference}
    +{}  {\PrintBook}                   {book}
    +{,} { }                            {booktitle}
    +{,} { \PrintDateB}                 {date}
    +{,} { pp.~}                        {pages}
    +{,} { }                            {status}
    +{,} { available at \eprint}        {eprint}
    +{}  { \parenthesize}               {language}
    +{}  { \PrintTranslation}           {translation}
    +{;} { \PrintReprint}               {reprint}
    +{.} { }                            {note}
    +{.} {}                             {transition}
    +{}  {\SentenceSpace \PrintReviews} {review}
}
\BibSpec{book}{%
    +{}  {\PrintPrimary}                {transition}
    +{,} { \TitleWithUrl}               {title}
    +{.} { }                            {part}
    +{:} { \textit}                     {subtitle}
    +{,} { \PrintEdition}               {edition}
    +{}  { \PrintEditorsB}              {editor}
    +{,} { \PrintTranslatorsC}          {translator}
    +{,} { \PrintContributions}         {contribution}
    +{,} { }                            {series}
    +{,} { \voltext}                    {volume}
    +{,} { }                            {publisher}
    +{,} { }                            {organization}
    +{,} { }                            {address}
    +{,} { \PrintDateB}                 {date}
    +{,} { }                            {status}
    +{}  { \parenthesize}               {language}
    +{}  { \PrintTranslation}           {translation}
    +{;} { \PrintReprint}               {reprint}
    +{.} { }                            {note}
    +{.} {}                             {transition}
    +{}  {\SentenceSpace \PrintReviews} {review}
}

\BibSpec{thesis}{%
    +{}  {\PrintAuthors}                {author}
    +{.} { \TitleWithUrl}               {title}
    +{:} { \textit}                     {subtitle}
    +{,} { \PrintThesisType}            {type}
    +{,} { }                            {organization}
    +{} { \PrintDate}                  {date}
    +{,} { }                            {address}
    +{,} { \eprint}                     {eprint}
    +{,} { }                            {status}
    +{}  { \parenthesize}               {language}
    +{}  { \PrintTranslation}           {translation}
    +{;} { \PrintReprint}               {reprint}
    +{.} { }                            {note}
    +{.} {}                             {transition}
    +{}  {\SentenceSpace \PrintReviews} {review}
}

\newtheorem{theorem}{Theorem}[section]
\newtheorem{lemma}[theorem]{Lemma}

\newtheorem{fact}[theorem]{Fact}

\theoremstyle{definition}
\newtheorem{definition}[theorem]{Definition}

\theoremstyle{remark}
\newtheorem{remark}[theorem]{Remark}

\numberwithin{equation}{section}

\newcommand{\qthreep}{\mathbb{Q}^3_+}

\newcommand{\bfOmega}{\mathbf{\Omega}}
\newcommand{\hthree}{\mathbb{H}^3(-1)}
\newcommand{\dsthree}{\mathbb{S}^3_1(1)}
\newcommand{\sltc}{\mathrm{SL}(2,\mathbb{C})}
\newcommand{\lfour}{\mathbb{L}^4}
\newcommand{\hermtwo}{\mathrm{Herm}(2,\mathbb{C})}

\newcommand{\mcL}{\mathcal{L}}

\title[Bj\"{o}rling problem for zero mean curvature surfaces in $\qthreep$]{Bj\"{o}rling problem for zero mean curvature surfaces in the three-dimensional light cone}

\author{Joseph Cho}
\address[Joseph Cho]{Institute of Discrete Mathematics and Geometry, TU Wien, Wien, 1040, Austria}
\email{jcho@geometrie.tuwien.ac.at}

\author{So Young Kim}
\address[So Young Kim]{Department of Mathematics, Korea University, Seoul, 02841, Republic of Korea}
\email{sykim0224@gmail.com}

\author{Dami Lee}
\address[Dami Lee]{Department of Mathematics, Indiana University, Bloomington, IN, 47405, USA}
\email{damilee@indiana.edu}

\author{Wonjoo Lee}
\address[Wonjoo Lee]{Department of Mathematics, Korea University, Seoul, 02841, Republic of Korea}
\email{wontail123@korea.ac.kr}

\author{Seong-Deog Yang}
\address[Seong-Deog Yang]{Department of Mathematics, Korea University, Seoul, 02841, Republic of Korea}
\email{sdyang@korea.ac.kr}

\subjclass[2020]{Primary 53A10; Secondary 53B30.}
\keywords{Zero mean curvature surfaces, Bj\"orling representation}
\begin{document}
\begin{abstract}
	We solve the Bj\"{o}rling problem for zero mean curvature surfaces in the three-dimensional light cone.
	As an application, we construct and classify all rotational zero mean curvature surfaces.
\end{abstract}

\maketitle

\section{Introduction}

The \emph{classical Bj\"{o}rling problem} \cite{bjorling_integrationem_1844} poses the following question: Suppose that a curve $\gamma : I \to \mathbb{E}^3$ and a unit vector field $N$ along $\gamma$ in the Euclidean 3-space $\mathbb{E}^3$ are given so that
	\[
		\dot{\gamma} \cdot N = 0, 
	\]
where $\dot{\phantom{\gamma}}$ denotes the differentiation with respect to $u \in I$, and $\cdot$ is the standard Euclidean inner product.
The goal is to find a minimal surface $X$ which contains $\gamma$ and whose unit normal along $\gamma$ is the prescribed $N$.
Due to the Weierstrass representation of minimal surfaces \cite{weierstrass_untersuchungen_1866}, this problem can be solved via analytic extensions of $\gamma$ and $N$ to the complex plane:
	\[
		X(z) = \int_{u_0}^z (\dot\gamma(w) - i N(w) \times \dot\gamma(w)) \dif{w}
	\]
where $\times$ denotes the standard Euclidean cross product.
The curve and the prescribed normal together are called the \emph{Björling data}.
Such Bj\"{o}rling type problems are well studied across various types of surfaces in various space forms (see, for example, \cite{alias_bjorling_2003, asperti_bjorling_2006, brander_bjorling_2010, brander_bjorling_2018, dussan_bjorling_2017, kim_spacelike_2011, kim_prescribing_2007, yang_bjorling_2017}).

The Björling data can be slightly modified by noting that $\mathcal{L} := N \times \dot{\gamma}$ is perpendicular to $\dot{\gamma}$ and of the same length to $\dot{\gamma}$.
Thus if $\mathcal{L}$ is prescribed then one can recover $N$ via $N = \frac{\dot{\gamma} \times \mathcal{L} }{ |\dot{\gamma} \times \mathcal{L}|}$, allowing us to adopt $\gamma$ and $\mathcal{L}$ instead of $\gamma$ and $N$ as Bj\"{o}rling data.
Geometrically, the vector field $\mathcal{L}$ of the Bj\"{o}rling data is equivalent to the prescription of the tangent vector field of the surface that is perpendicular to $\dot\gamma$ along $\gamma$.
This change in viewpoint of the prescribed data for the Björling problem proved to be useful for obtaining the \emph{(singular) Bj\"{o}rling representation}  \cite{kim_prescribing_2007} for maxfaces \cite{umehara_maximal_2006} and generalized timelike minimal surfaces \cite{kim_spacelike_2011} in Lorentz $3$-space, and also the Bj\"{o}rling representation for zero mean curvature surfaces in isotropic $3$-space \cite{seo_zero_2021}.

Our goal of the paper is to solve the Bj\"{o}rling problem for zero mean curvature surfaces in the $3$-dimensional light cone $\qthreep$.
In the similar cases of other quadrics of Lorentz $4$-space, namely, the hyperbolic $3$-space $\hthree$ and de Sitter $3$-space $\dsthree$, the tangent space is isomorphic to either the Euclidean $3$-space or the Lorentz $3$-space, respectively.
Thus one can use the cross product in the tangent space to obtain Bj\"{o}rling representations \cite{yang_bjorling_2017} in an analogous manner to the Euclidean case.
However, the tangent space to $\qthreep$ is isomorphic to isotropic $3$-space, so that it does not have a cross product structure.
Therefore, for the Bj\"{o}rling problem in $\qthreep$, we use the alternative viewpoint and assume that the given Bj\"{o}rling data consists of a spacelike analytic curve $\gamma$, together with a spacelike vector field $\mathcal{L}$ along the curve, and find the zero mean curvature surface which contains $\gamma$ and has $\mathcal{L}$ as a tangent vector field along $\gamma$.

The paper is structured as follows: After reviewing the basic geometry and surface theory of $\qthreep$ in Section~\ref{sec:prem}, we describe the process to solve the Bj\"{o}rling problem for a given spacelike analytic curve with prescribed tangent vector field in Section~\ref{Sec:202302260350PM}, culminating in the \emph{Bj\"orling representation for zero mean curvature surfaces in $3$-dimensional light cone} (see Theorem~\ref{thm:main}).
As an application, we construct and classify (see Theorem~\ref{thm:class}) rotationally invariant zero mean curvature surfaces in $\qthreep$ in Section~\ref{Sec:202302260929PM}.
Furthermore, many surfaces admitting Weierstrass representations in various space forms with indefinite metric can be extended across lightlike lines (see, for example, \cite{akamine_reflection_2021, akamine_reflection_2022, akamine_space-like_2019, fujimori_analytic_2022, fujimori_zero_2015-1, umehara_hypersurfaces_2019}); we give an example of such surface in $\qthreep$ in Section~\ref{sec:analytic}.

\section{Preliminaries}\label{sec:prem}
We first briefly review the geometry of three-dimensional light cone, and the surface theory within. For a detailed description, see \cite{liu_surfaces_2007, liu_hypersurfaces_2008, liu_representation_2011}.

\subsection{Hermitian matrix model of three-dimensional light cone}
Let $\lfour$ denote the Lorentz $4$-space, with inner product
	\[
		\langle (t_1,x_1,y_1,z_1), (t_2,x_2,y_2,z_2) \rangle = -t_1 t_2 + x_1 x_2 + y_1 y_2 + z_1 z_2.
	\]
The Lorentz $4$-space can be identified with the set of $2\times2$ Hermitian matrices $\hermtwo$ via
	\[
		(t,x,y,z) \sim \begin{pmatrix} t+z & x+iy \\ x-iy & t-z \end{pmatrix}
	\]
where we will abuse notation between vectors and matrices.
Then for any $V, W \in \hermtwo \cong \lfour$,
	\[
		\langle V, W \rangle = -\frac{1}{2} \left(\det{(V+W)} - \det{V} - \det{W} \right),
	\]
so that
	\[
		|V|^2 := \langle V, V \rangle = -\det{V}.
	\]
We note here that the symmetric bilinear form $\langle \cdot , \cdot \rangle$ is well-defined for all $A \in \mathrm{M}(2, \mathbb{C})$.

Defining
	\[
		f_0 := \begin{pmatrix} 0 & 0 \\ 0 & -2 \end{pmatrix}, \quad
		f_1 := \begin{pmatrix} 0 & 1 \\ 1 & 0 \end{pmatrix}, \quad
		f_2 := \begin{pmatrix} 0 & i \\ -i & 0 \end{pmatrix}, \quad
		f_3 := \begin{pmatrix} 1 & 0 \\ 0 & 0 \end{pmatrix},
	\]
we see that $\{ f_0, f_1, f_2, f_3\}$ form an asymptotically orthonormal basis of $\mathbb{L}^4 \cong \hermtwo$.

In the Hermitian model, hyperbolic $3$-space $\hthree$, $3$-dimensional light cone $\qthreep$, and de Sitter $3$-space $\dsthree$ can be defined as quadrics in $\lfour$ via
	\begin{align*}
		\hthree &:= \{ X \in \hermtwo :  \langle X, X \rangle = -1, \operatorname{tr}{X}>0\}, \\
		\qthreep &:= \{ X \in \hermtwo :  \langle X, X \rangle = 0, \operatorname{tr}{X}>0\}, \\
		\dsthree &:= \{ X \in \hermtwo :  \langle X, X \rangle = 1 \},
	\end{align*}
respectively.
Note that for any $X \in \qthreep$, there is some $F \in \sltc$ such that
	\[
		X = F \begin{pmatrix} 1 & 0 \\ 0 & 0 \end{pmatrix} F^\star = F f_3 F^\star
	\]
where $F^\star$ denotes the conjugate transpose of $F$.

When we visualize surfaces in $\qthreep$, we will use the following stereographic projection
	\[
		\qthreep \ni (t, x,y,z)  \mapsto \left( \frac{x}{1+t}, \frac{y}{1+t}, \frac{z}{1+t} \right).
	\]
Then $\qthreep$ is identified with $\{ (a,b,c) \in \mathbb{R}^3 :  0 < a^2 + b^2 + c^2 < 1 \}$.

\subsection{Surface theory and Weierstrass-type representation}
Let $D$ be a two-dimensional simply-connected domain, and suppose $X : D \to \qthreep$ is a spacelike immersion, that is, the induced metric on the tangent plane of $X$ at every point $p \in D$ is Riemannian.
Then for conformal coordinates $(u,v) \in D$ with complex structure given via $z = u + iv$, suppose that the first fundamental form is given by 
	\[
		\dif{s}^2 = \phi^2 (\dif{u}^2 + \dif{v}^2) = \phi^2 \dif{z}\dif{\bar{z}}
	\]
for some $\phi: D \to \mathbb{R}^\times$.
Furthermore, we have $\langle X, X_u \rangle = \langle X, X_v \rangle = 0$ so that there exists a unique lightlike $n : D \to \lfour$ such that
	\begin{equation}\label{eqn:gauss}
		\langle n, n \rangle = \langle n, X_u \rangle = \langle n, X_v \rangle = 0, \quad \langle n, X \rangle = 1.
	\end{equation}
Such $n$ is called the \emph{Gauss map} of $X$, and the coefficients of the second fundamental form is then computed in terms of $n$ via
	\[
		L := -\langle X_u, n_u \rangle, \quad M := -\langle X_u, n_v \rangle = -\langle X_v, n_u \rangle,\quad N := -\langle X_v, n_v \rangle.
	\]
Hence, the shape operator $S$ is
	\[
		S = \phi^{-2}\begin{pmatrix} L & M \\ M & N \end{pmatrix}
	\]
with the (extrinsic) Gaussian curvature $K$ and mean curvature $H$ given by
	\[
		K = \det{S}, \quad H = \frac{1}{2}\operatorname{tr} S.
	\]
Those surfaces $X$ with $H \equiv 0$ will be referred to as \emph{zero mean curvature surfaces}.

We recall the Weierstrass-type representation for zero mean curvature surfaces in the three-dimensional light cone \cite[Theorem~3.2]{liu_representation_2011} (see also \cite[\S~4.2]{pember_weierstrass-type_2020} and \cite[Theorem~39]{seo_zero_2021}):
\begin{fact}\label{fact:Wrep}
	Any zero mean curvature surface $X : D \to \qthreep$ can locally be represented as
		\[
			X = F f_3 F^\star = F \begin{pmatrix} 1 & 0 \\ 0 & 0 \end{pmatrix} F^\star
		\]
	where $F : D \to \sltc$ satisfies
		\[
			\dif{F} F^{-1} = \begin{pmatrix} G & -G^2 \\ 1 & -G \end{pmatrix}\Omega
		\]
	for some meromorphic function $G : D \to \mathbb{C}$ and holomorphic $1$-form $\Omega : D \to \mathbb{C}$ such that $G^2\Omega$ is holomorphic.
	The pair $(G, \Omega)$ is called the \emph{Weierstrass data}.
\end{fact}

\section{Solution of the Bj\"{o}rling problem}\label{Sec:202302260350PM}
Using the Weierstrass-type representation in Fact~\ref{fact:Wrep}, we will now solve the Björling problem for zero mean curvature surfaces in $\qthreep$.
Unlike the cases of hyperbolic $3$-space $\hthree$ and de Sitter $3$-space $\dsthree$, the Bj\"{o}rling data we consider will be an analytic curve together with a tangent vector field due to the lack of cross product structure in the tangent space of $\qthreep$.

\subsection{The conformality condition for the Bj\"orling data} 
For an interval $I$ with parameter $u \in I$, let $\gamma : I \to \qthreep$ be a spacelike analytic curve and $\mathcal{L}$ be an analytic spacelike vector field along $\gamma$ such that
	\begin{equation}\label{Eq:202301221105PM}
		\langle \dot{\gamma}, \dot{\gamma} \rangle = \langle \mathcal{L}, \mathcal{L} \rangle, \quad
		\langle \dot{\gamma}, \mathcal{L} \rangle = 0, \quad
		\langle \gamma, \mathcal{L} \rangle = 0.
	\end{equation}
Geometrically, the conditions \eqref{Eq:202301221105PM} ensure that $\mathcal{L}$ will be a tangent vector field; thus, we refer to \eqref{Eq:202301221105PM} as the \emph{conformality condition}.
Our goal is to find a zero mean curvature surface $X : D \to \qthreep$ which contains $\gamma$ and which has $\mathcal{L}$ as a tangent vector field along $\gamma$.
%
%

Note that if there is a such a zero mean curvature surface $X$, then 
	\begin{equation}\label{Eq:202301221148PM}
		\dif{X} = \dif{F} f_3 F^\star = \dif{F} F^{-1} F  f_3 F^\star = \begin{pmatrix} G & -G^2 \\ 1 & -G \end{pmatrix}\Omega X =: \bfOmega X.
	\end{equation}
Thus, from $\gamma$ and $\mathcal{L}$, we will construct such $1$-form $\bfOmega$.	

\begin{remark}
	An important distinction from the cases of $\hthree$ and $\dsthree$ is that $X \in \qthreep$ is not invertible since $\det X = 0$.
	Thus, instead of considering $\dif{X} X^{-1}$ as in the cases of $\hthree$ and $\dsthree$, we consider the equation as in \eqref{Eq:202301221148PM}.
\end{remark}

For Bj\"{o}rling data $\gamma$ and $\mathcal{L}$ satisfying conformality condition \eqref{Eq:202301221105PM}, let $\gamma(u) = X(u,0)$ and $\mathcal{L}(u) := X_v(u,0)$.
Then on one dimensional domain $I$, \eqref{Eq:202301221148PM} implies that we must solve for $\bfOmega =: \tilde\Omega \dif{u}$ satisfying
	\begin{equation}\label{Eq:202301230639AM}
		\Lambda(u) := \frac{1}{2} (\dot{\gamma}(u) - i \mathcal{L}(u)) = X_z (u,0) = \tilde\Omega(u) X(u,0) = \tilde\Omega(u) \gamma(u).
	\end{equation}
We claim that there is a unique $\tilde\Omega$ which solves \eqref{Eq:202301230639AM}.

Note that $\det{\Lambda}=0$ because of \eqref{Eq:202301221105PM}, and that $\det{\gamma}=0$ since it is a curve in $\qthreep$.
By standard theory we know that 
	\[
		\det{\tilde\Omega} = \operatorname{tr}{\tilde\Omega} =0.
	\]
Therefore 
	\[
		\bfOmega = \begin{pmatrix} G & -G^2 \\ 1 & -G \end{pmatrix} \Omega
	\]
for some functions $G$ and $1$-form $\Omega$.
We rewrite \eqref{Eq:202301230639AM} as
	\begin{equation}\label{Eq:202301230644AM}
		\begin{pmatrix} \Lambda_{11} & \Lambda_{12} \\ \Lambda_{21} & \Lambda_{22} \end{pmatrix}
			= \begin{pmatrix} G \tilde\Omega& -G^2 \tilde\Omega\\ \tilde\Omega & -G \tilde\Omega \end{pmatrix} 
			\begin{pmatrix} \gamma_{11} & \gamma_{12}  \\ \gamma_{21}  & \gamma_{22} \end{pmatrix}.
	\end{equation}
By examining $(1,1)$ and $(2,1)$ components of \eqref{Eq:202301230644AM}, we see that $G=G_1$ and $\tilde\Omega = \tilde\Omega_1$ where
	\begin{equation}\label{Eq:202301220834PM}
		G_1 = \frac{\Lambda_{11}}{\Lambda_{21}}, \quad
		\tilde\Omega_1 = \frac{\Lambda_{21}^2}{ \Lambda_{21} \gamma_{11} - \Lambda_{11} \gamma_{21}}. 
	\end{equation}
On the other hand, by examining $(1,2)$ and $(2,2)$ components of \eqref{Eq:202301230644AM}, we see that $G=G_2$ and $\tilde\Omega = \tilde\Omega_2$ where 
	\begin{equation}\label{Eq:202301220835PM}
		G_2 = \frac{\Lambda_{12}}{\Lambda_{22}}, \quad
		\tilde\Omega_2 = \frac{\Lambda_{22}^2}{ \Lambda_{22} \gamma_{12} - \Lambda_{12} \gamma_{22}}. 
	\end{equation}

It is immediate to see $G_1=G_2$ from the fact that $\det{\Lambda} =0$.
Then
	\begin{align*}
		\tilde\Omega_1 &= \frac{\Lambda_{21}^2}{ \Lambda_{21} \gamma_{11} - \frac{\Lambda_{12} \Lambda_{21}}{\Lambda_{22}} \gamma_{21}}
			= \frac{\Lambda_{21} \Lambda_{22}}{ \Lambda_{22} \gamma_{11} - \Lambda_{12} \gamma_{21}},\\
		\tilde\Omega_2  &= \frac{\Lambda_{22}^2}{ \Lambda_{22} \gamma_{12} - \frac{\Lambda_{11} \Lambda_{22}}{\Lambda_{21}} \gamma_{22}}
			= \frac{\Lambda_{21} \Lambda_{22}}{\Lambda_{21} \gamma_{12} - \Lambda_{11} \gamma_{22}}.
	\end{align*}
Hence, it is enough to show that the denominators are the same.

For this matter, we multiply \eqref{Eq:202301230644AM} with the adjunct of $\gamma$ from the right to obtain
	\[
		\begin{pmatrix} \Lambda_{11} & \Lambda_{12} \\ \Lambda_{21} & \Lambda_{22} \end{pmatrix}
		\begin{pmatrix} \gamma_{22} & -\gamma_{12}  \\ - \gamma_{21} & \gamma_{11} \end{pmatrix}
			= \begin{pmatrix} G \tilde\Omega& -G^2 \tilde\Omega\\ \tilde\Omega & -G \tilde\Omega \end{pmatrix}
			\begin{pmatrix} \det{\gamma} & 0 \\ 0 & \det{\gamma} \end{pmatrix}.
	\]
However, we have that $\det{\gamma} =0$ since it is a curve in $\qthreep$.
Therefore, the left hand side is a zero matrix.
In particular, its trace is 0; hence
	\[
		\Lambda_{11} \gamma_{22} - \Lambda_{12} \gamma_{21} -\Lambda_{21} \gamma_{12} + \Lambda_{22} \gamma_{11} = 0,
	\]
implying that $\tilde\Omega_1 = \tilde\Omega_2$.

Now, the expressions for $\tilde\Omega_1$ and $\tilde\Omega_2$ in \eqref{Eq:202301220834PM} and \eqref{Eq:202301220835PM} implies that the Bj\"{o}rling data $\gamma$ and $\mathcal{L}$ satisfying \eqref{Eq:202301221105PM} must additionally satisfy 
	\begin{equation}\label{eqn:weirdCond}
		\gamma_{11} \Lambda_{21} - \gamma_{21} \Lambda_{11} \neq 0.
	\end{equation}
Unfortunately, the conformality conditions \eqref{Eq:202301221105PM} alone does not guarantee this; one can even construct explicit examples where the expression \eqref{eqn:weirdCond} vanishes for some Bj\"orling data satisfying the conformality conditions.

\subsection{The orientability condition for the Bj\"orling data} 
To give a geometric interpretation of the condition \eqref{eqn:weirdCond}, we define the following notion.
\begin{definition}
	For two linearly independent spacelike vectors $U, V \in \lfour$ spanning a Riemannian subspace $W$ of $\lfour$, let $\{e_1, e_2\}$ be an orthonormal basis of $W$ such that there is some $F \in \sltc$ with $F f_i F^{\star} = e_i$ for $i = 1,2$.
	The \emph{signed area of $U$ and $V$}, denoted by $\mathrm{SA}(U,V)$, is defined via
		\[
			U \wedge V =: \mathrm{SA}(U,V) e_1 \wedge e_2.
		\]
\end{definition}

An important feature of the signed area is the fact that it is invariant under \emph{orientation-preserving isometries} of $\qthreep$.
On the other hand, the square of the signed area (or the squared area) can be calculated using the inner product of the ambient Lorentz $4$-space:
\begin{lemma}\label{lemma:sA}
	For linearly independent spacelike vectors $U,V$ spanning a Riemannian subspace $W$, we have
		\[
			 \mathrm{SA}(U,V)^2 = \langle U, U \rangle \langle V, V \rangle - \langle U, V \rangle^2.
		\]
\end{lemma}
\begin{proof}
	Let $\{e_1, e_2\}$ be an orthonormal basis $W$ such that $F f_i F^{\star} = e_i$ for $i = 1,2$ with $F \in \sltc$.
	Writing
		\[
			U = a e_1 + b e_2, \quad V = c e_1 + d e_2
		\]
	for some constants $a,b,c,d \in \mathbb{R}$, we calculate that
		\[
			\mathrm{SA}(U,V) e_1 \wedge e_2 = (a e_1 + b e_2) \wedge (c e_1 + d e_2) = (ad - bc) e_1 \wedge e_2.
		\]
	On the other hand, we have
		\[
			\langle U, U \rangle \langle V, V \rangle - \langle U, V \rangle^2 = (a^2 + b^2)(c^2 + d^2) - (ac + bd)^2 = (ad - bc)^2,
		\]
	giving us the desired conclusion.
\end{proof}

The next lemmata give additional geometric criteria for the Bj\"{o}rling data to satisfy \eqref{eqn:weirdCond} in terms of the signed area.
\begin{lemma}
	Given Bj\"{o}rling data $\gamma$ and $\mathcal{L}$ satisfying the conformality condition, the signed area of $\dot\gamma$ and $\mathcal{L}$ does not vanish.
\end{lemma}
\begin{proof}
	Suppose for contradiction that $\mathrm{SA}(\dot{\gamma}, \mcL)=0$.
	Then direct calculations show that $\mathcal{\mcL} = \alpha \gamma \pm \dot{\gamma}$ for some $\alpha$.
	Thus
		\[
			\pm \langle \dot{\gamma}, \dot{\gamma} \rangle = \langle \dot{\gamma}, \alpha \gamma \pm \dot{\gamma} \rangle =  \langle  \dot{\gamma}, \mathcal{L} \rangle =0,
		\]
which is a contradiction since $\dot{\gamma}$ is spacelike.
\end{proof}

\begin{lemma}\label{lemma:orientation}
	Given Bj\"{o}rling data $\gamma$ and $\mathcal{L}$ satisfying the conformality condition, $G$ and $\Omega$ are well-defined if and only if the signed area of $\dot{\gamma}$ and $\mathcal{L}$ is negative.
\end{lemma}
\begin{proof}
	Fix any $u_0 \in I$ throughout the proof.
	Since signed area is invariant under orientation-preserving isometries, we may assume without loss of generality that $\gamma = 2 f_3$, so that any vector in $T_{\gamma} \qthreep$ takes the form
		\[
			\ell \gamma + x f_1 + y f_2
		\]
	for some $\ell, x, y \in \mathbb{R}$.
	By the conformality conditions \eqref{Eq:202301221105PM} on the Bj\"{o}rling data, we have that $\dot{\gamma}, \mathcal{L} \in T_{\gamma} \qthreep$; therefore, there are some constants $a,b,c,d,e,f \in \mathbb{R}$ such that
		\[
			\dot{\gamma} = a \gamma + b f_1 + c f_2, \quad\text{and}\quad
			\mathcal{L} = d \gamma +  e f_1 + f f_2.
		\]
	Since $\Lambda := \frac{1}{2} \left(  \dot{\gamma} - i \mathcal{L}\right)$, we calculate that
		\begin{align*}
			\gamma_{11} \Lambda_{21}  - \gamma_{21} \Lambda_{11} = (b - f) - i (c + e),
		\end{align*}
	so that
		\begin{equation}\label{Eq:202302230818PM}
			\begin{aligned}
				|\gamma_{11} \Lambda_{21}  - \gamma_{21} \Lambda_{11}|^2 
					&= b^2 + c^2 + e^2 + f^2 + 2 (c e  - b f)   \\
					&= |\dot{\gamma}|^2 + |\mathcal{L}|^2 + 2 (c e  - b f) \\
					&= 2 (|\dot{\gamma}|^2 - (b f - ce)).
			\end{aligned}
		\end{equation}
	
	Now since $\gamma$-component does not affect the signed area, we note that
		\[
			\mathrm{SA}(\dot{\gamma}, \mathcal{L}) = \mathrm{SA}(b f_1 + c f_2, e f_1 + f f_2).
		\]
	However, we can also calculate that
		\[
			\mathrm{SA}(b f_1 + c f_2, e f_1 + f f_2) f_1 \wedge f_2
				= (b f_1 + c f_2) \wedge (e f_1 + f f_2)
				= (bf - ce) f_1 \wedge f_2,
		\]
	allowing us to deduce that
		\[
			\mathrm{SA}(\dot{\gamma}, \mathcal{L}) = bf - ce.
		\]
	On the other hand, Lemma~\ref{lemma:sA} and the conformality conditions \eqref{Eq:202301221105PM} imply that
		\[
			\mathrm{SA}(\dot{\gamma}, \mathcal{L})^2 = |\dot\gamma|^2 |\mathcal{L}|^2 - \langle \dot\gamma, \mathcal{L} \rangle^2 = | \dot\gamma|^4
 		\]
	so that
		\[
			bf - ce = \mathrm{SA}(\dot{\gamma}, \mathcal{L}) = 
				\begin{cases}
					|\dot\gamma|^2,& \text{if } \mathrm{SA}(\dot{\gamma}, \mathcal{L})>0, \\
					-|\dot\gamma|^2,& \text{if } \mathrm{SA}(\dot{\gamma}, \mathcal{L})<0.
				\end{cases}
		\]
	Thus, using \eqref{Eq:202302230818PM} allows us to conclude that
		\[
			|\gamma_{11} \Lambda_{21}  - \gamma_{21} \Lambda_{11}|^2  = 
				\begin{cases} 
					0,& \text{if } \mathrm{SA}(\dot{\gamma}, \mathcal{L})>0, \\
					4|\dot{\gamma}|^2,& \text{if } \mathrm{SA}(\dot{\gamma}, \mathcal{L})<0.
				\end{cases}
		\]
	Therefore, $G$ and $\Omega$ are well-defined if and only if $SA(\dot{\gamma}, \mcL)<0$.
\end{proof}

\begin{remark}
	The signed area condition of Lemma~\ref{lemma:orientation} encodes the fact that the orientation of $\dot{\gamma}$ and $\mathcal{L}$ must be correct.
	For this reason, we call the signed area condition, the \emph{orientability condition}.
\end{remark}

\begin{remark}
	The orientability condition is unnecessary for the Bj\"orling data in other quadrics such as $\hthree$ and $\dsthree$; the crucial difference arises in the definition of the (lightlike) Gauss map used to define the second fundamental form.
	In the case of $\hthree$ and $\dsthree$, the Gauss map $n$ and the position vector $X$ must satisfy $\langle n, X \rangle = 0$; however, in the case $\qthreep$, we have that $\langle n, X \rangle = 1$, so that one cannot switch the signs on the Gauss map freely.
\end{remark}

We summarize our results in the next main theorem of our paper, the \emph{Bj\"{o}rling representation for zero mean curvature surfaces in the $3$-dimensional light cone}:
\begin{theorem}\label{thm:main}
	Consider $\qthreep$ as a subset of $\hermtwo$ identified with $\lfour$.
	Given a spacelike analytic curve $\gamma : I  \to \qthreep$ and an analytic vector field $\mathcal{L}$ satisfying the conformality condition
	\[
		\langle \dot{\gamma}, \dot{\gamma} \rangle = \langle \mathcal{L}, \mathcal{L} \rangle, \quad
		\langle \dot{\gamma}, \mathcal{L} \rangle = 0, \quad
		\langle \gamma, \mathcal{L} \rangle = 0,
	\]
	and the orientability condition
	\[
		\mathrm{SA}( \dot{\gamma}(u), \mcL(u)) <0
		\qquad\text{for all }
		u \in I,
	\]
	there exists a unique zero mean curvature surface $X$ such that
	\[
		X(u,0) = \gamma(u), \qquad X_v(u,0) = \mathcal{L}(u)
	\]
	for all $u \in I$.
	It is given by $X = F f_3 F^\star$ where $F \in C^\omega(\mathcal{U} \subset \mathbb{C}, \sltc)$ is a solution to 
	\[
		\dif{F} F^{-1} = \text{ the analytic extension of } \bfOmega,
		\qquad
		F(u_0) f_3 F(u_0)^\star = \gamma(u_0)
	\]
	where $u_0$ is an element of the domain of $\gamma$ and $\bfOmega =: \tilde\Omega \dif{u}$ is the solution of
	\[
		\Lambda = \tilde\Omega \gamma,
		\qquad\text{where}\quad
		\Lambda(u) := \frac{1}{2} (\dot{\gamma}(u) - i \mathcal{L}(u))  \in M(2,\mathbb{C}).
	\]
\end{theorem}	
	
\section{Rotational zero mean curvature surfaces in \texorpdfstring{$\qthreep$}{Q3+}}\label{Sec:202302260929PM}
In this section we construct and classify rotationally invariant zero mean curvature surfaces, which we call \emph{catenoids} of the $3$-dimensional light cone.
The Bj\"{o}rling problem is especially apt for the construction of such surface as the initial curve can be selected as the orbit of a single point under rotations of $\qthreep$.

First, we define rotational surfaces in $\qthreep$ analogously to the definition given in hyperbolic spaces \cite[Definition~2.2]{do_carmo_rotation_1983} (see also \cite[Definition~3.1]{abe_constant_2018}):
\begin{definition}
	Let $P^k$ denote a $k$-dimensional subspace of $\lfour$.
	Choosing some $P^2$ and $P^3 \supset P^2$ such that $P^3 \cap \qthreep \neq \varnothing$, let $\mathrm{O}(P^2)$ denote the set of orthogonal transformations that leave $P^2$ fixed. 
	For a regular curve $\gamma$ in $P^3 \cap \qthreep$ that does not meet $P^2$, we call the orbit of $\gamma$ under the action of  $\mathrm{O}(P^2)$ a \emph{rotational surface generated by $\gamma$}.
	Furthermore, we say that the rotational surface is
		\begin{itemize}
			\item an \emph{elliptic} rotational surface if the induced metric on $P^2$ is Lorentzian,
			\item a \emph{parabolic} rotational surface if the induced metric on $P^2$ is degenerate, and
			\item a \emph{hyperbolic} rotational surface if the induced metric on $P^2$ is Riemannian.
		\end{itemize}
\end{definition}

The following facts will be useful in recovering the explicit parametrizations of the catenoids from the rotationally invariant Bj\"{o}rling data:
\begin{fact}[{\cite[Lemma~4.2]{yang_bjorling_2017}}]\label{fact:sol1}
	For $\nu \in \mathbb{C}\setminus\{0\}$,
		\[
			F_0(z) := 
				\frac{1}{2}\begin{pmatrix} z^{1/2} & 0 \\ 0 & z^{-1/2} \end{pmatrix}
				\begin{pmatrix}  \sqrt{\nu}+\sqrt{\nu}^{-1} & \sqrt{\nu}-\sqrt{\nu}^{-1} \\ \sqrt{\nu}-\sqrt{\nu}^{-1} & \sqrt{\nu}+\sqrt{\nu}^{-1} \end{pmatrix}
				\begin{pmatrix} z^{-\nu/2} & 0 \\ 0 & z^{\nu/2} \end{pmatrix}
		\]
	is a solution to
		\begin{equation}\label{Eq:202211170549AM}
			\dif{F}F^{-1} = \begin{pmatrix} G & -G^2 \\ 1 & -G \end{pmatrix} \Omega  
			\quad\text{with}\quad 
			G := z,   \quad \Omega := \lambda \frac{\dif{z}}{z^2}, \quad \lambda :=\frac{1-\nu^2}{4}.
		\end{equation}
	where $\lambda \in \mathbb{C}\setminus\{\frac{1}{4} \}$.
	Furthermore, $\tilde{F} := F_0 R$ also solves \eqref{Eq:202211170549AM} where $R \in \sltc$ controls the initial condition.
\end{fact}

\begin{fact}[{\cite[Lemma~4.2]{yang_bjorling_2017}}]\label{fact:sol2}
	For $\beta \in \mathbb{C}\setminus\{0\}$, let 
		\[
			F_1(z) := \begin{pmatrix} 1 & z \\ 0 & 1 \end{pmatrix}
				\begin{pmatrix} \beta^{-1/2} & 0 \\ 0 & \beta^{1/2} \end{pmatrix}
				\begin{pmatrix} \cos{\beta z} & -   \sin{\beta z}\\
				 \sin{\beta z} & \cos{\beta z} \end{pmatrix}
				\begin{pmatrix} \beta^{1/2} & 0 \\ 0 & \beta^{-1/2} \end{pmatrix}
		\]
	is a solution to
		\begin{equation}\label{Eq:202301251205AM2}
			 \dif{F} F^{-1} = \begin{pmatrix} G & -G^2 \\ 1 & -G \end{pmatrix} \Omega\quad\text{with}\quad 
			G := z,   \quad \Omega := \beta^2 \dif{z}.
		\end{equation}
	Furthermore, $\tilde{F} := F_1 R$ also solves \eqref{Eq:202301251205AM2} where $R \in \sltc$ controls the initial condition.
\end{fact}

\subsection{Elliptic catenoids}
Consider an elliptic rotation
	\[
		R^\mathrm{E}(u) = \begin{pmatrix} e^{iu} & 0 \\ 0 & e^{-iu} \end{pmatrix}
	\]
applied to the point $\begin{psmallmatrix} 1 & 1 \\ 1 & 1 \end{psmallmatrix} \in \qthreep$ to obtain an elliptic circle
	\[
		\gamma(u) = 
			R^\mathrm{E} (u)
			\begin{pmatrix} 1 & 1 \\ 1 & 1 \end{pmatrix}
			R^\mathrm{E} (u)^\star
			= \begin{pmatrix} 1 & e^{2iu} \\ e^{-2iu} & 1 \end{pmatrix}.
	\]
Note that $\langle \dot{\gamma}, \dot{\gamma} \rangle = 4$ and that $\gamma(0) = \begin{psmallmatrix} 1 & 0 \\ 1 & 1 \end{psmallmatrix} f_3 \begin{psmallmatrix} 1 & 0 \\ 1 & 1 \end{psmallmatrix}^\star$.

To complete the Bj\"{o}rling data, we first note that any $\mathcal{L}$ satisfying the conformality condition must be of the form
	\[
		\mathcal{L}_\pm(u) = f(u) \gamma(u) \pm 2 e_3
	\]
for any function $f$.
Now, we can check that $\mathcal{L}_+ = f \gamma + 2 e_3$ fails the orientability condition, that is,
	\[
		\Lambda_{21} \gamma_{11} - \Lambda_{11} \gamma_{21} = 0, 
	\]
while $\mathcal{L}_- = f \gamma - 2 e_3$ satisfies the orientability condition.

Thus we take $\mathcal{L} = \mathcal{L}_-$, and note that for rotationally invariant surfaces, we must have that $\mathcal{L}$ is generated by the elliptic rotation under consideration, i.e.
	\[
		\mathcal{L}(u) = R^\mathrm{E} (u) \mathcal{L}(0) R^\mathrm{E} (u)^\star.
	\]
Thus, we find that
	\[
		f(u) \gamma(u) - 2 e_3 = \mathcal{L} (u) = R^\mathrm{E} (u) \mathcal{L}(0) R^\mathrm{E} (u)^\star = f(0) \gamma(u) - 2 e_3,
	\]
and we have $f(u) = f(0) =: a$ is a constant function so that
	\[
		\mathcal{L}(u) = a \gamma(u) - 2 e_3.
	\]

Considering $\gamma$ and $\mathcal{L}$ as the Bj\"{o}rling data, we use the Bj\"{o}rling representation in Theorem~\ref{thm:main} to calculate that
	\[
		G = \frac{a-2}{a+2} e^{2iu}, \quad\text{and}\quad \Omega = \frac{-i(a+2)^2}{8} e^{-2iu} \dif{u}.
	\]

To obtain explicit parametrizations, we analytically extend $G$ and $\Omega$ and change the variables via
	\[
		\frac{a-2}{a+2} e^{2iw} =: z.
	\]
Then
	\[
		G= z, \quad
		\Omega = \lambda \frac{\dif{z}}{z^2}, \quad
		\lambda := \frac{4-a^2}{16},
	\]
allowing us to use Fact~\ref{fact:sol1} to compute the explicit parametrizations for the elliptic catenoid:
	\[
		X^{\mathrm{E}}
			= \begin{pmatrix}
				e^{(a-2)v} & e^{2iu + av} \\
				e^{-2iu + av} & e^{(a+2)v}
			\end{pmatrix}
			= R^\mathrm{E}(u) e^{av}
			\begin{pmatrix}
				e^{-2v} & 1 \\
				1 & e^{2v}
			\end{pmatrix}
			R^\mathrm{E}(u)^\star
	\]
where $w = u + i v$ (see also Figure~\ref{Ecatenoid}).

\begin{figure}
	\centering
	\begin{minipage}{0.46\linewidth}
		\centering
		\includegraphics[width=0.8\textwidth]{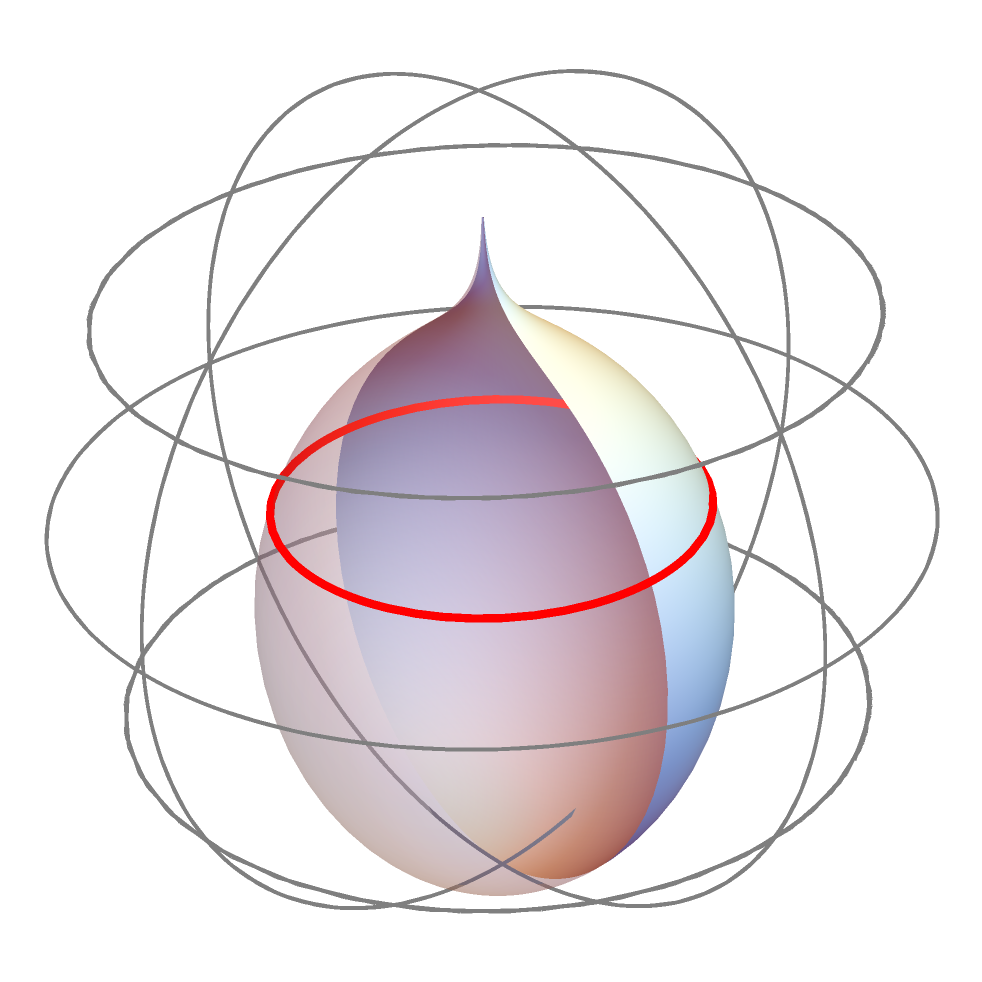}
	\end{minipage}
	\begin{minipage}{0.46\linewidth}
		\centering
		\includegraphics[width=0.8\textwidth]{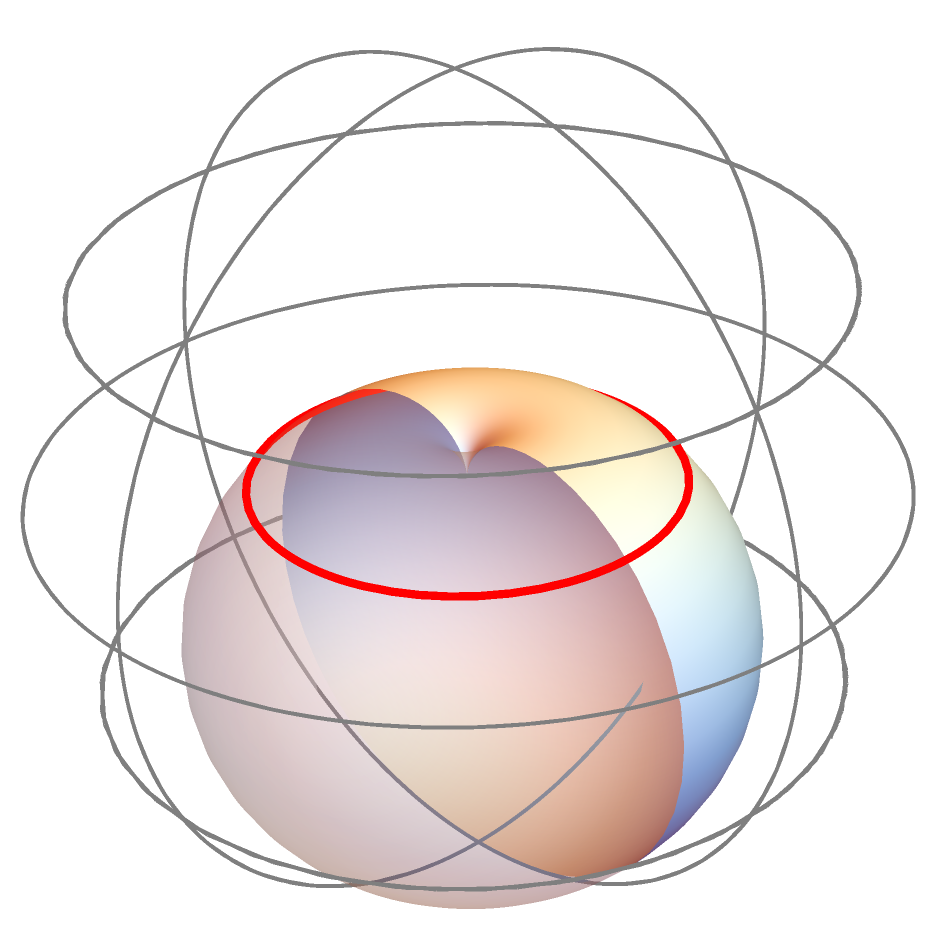}
	\end{minipage}
    \caption{Elliptic catenoids in $3$-dimensional light cone with $a = \frac{3}{2}$ on the left, and $a = 4$ on the right, where the initial given curve is highlighted.}
    \label{Ecatenoid}  
\end{figure}

\subsection{Hyperbolic catenoids}
Given a one parameter group of hyperbolic rotations
	\[
		R^\mathrm{H}(u) = \begin{pmatrix} e^u & 0 \\ 0 & e^{-u} \end{pmatrix},
	\]
the curve
	\[
		\gamma(u) := 
			R^\mathrm{H}(u)
			\begin{pmatrix} 1 & 1 \\ 1 & 1 \end{pmatrix}
			R^\mathrm{H}(u)^\star
			= \begin{pmatrix} e^{2u} & 1 \\ 1 & e^{-2u} \end{pmatrix}
	\]
is a hyperbolic circle.
Note that $\langle \dot{\gamma}, \dot{\gamma} \rangle = 4$ and that $\gamma(0) = \begin{psmallmatrix} 1 & 0 \\ 1 & 1 \end{psmallmatrix} f_3 \begin{psmallmatrix} 1 & 0 \\ 1 & 1 \end{psmallmatrix}^\star$.

To find the Björling data, we find that any $\mathcal{L}$ satisfying the conformality condition must take the form
	\[
		\mathcal{L}_\pm(u) = f(u)\gamma(u) \pm f_2,
	\]
for any function $f$.
Since $\mathcal{L}_+$ fails the orientability condition, we set $\mathcal{L} = \mathcal{L}_-$.
We also have that $\mathcal{L}$ is generated by the rotation under consideration; thus,
	\[
		\mathcal{L}(u) = b \gamma(u) \pm f_2,
	\]
for some constant $b$.

Now, we calculate the Bj\"{o}rling data as
	\[
		G = \frac{b+2i}{b-2i} e^{2u}, \quad\text{and}\quad \Omega = \frac{1}{8} (b-2i)^2 e^{-2u} \dif{u}.
	\]
Analytically extending $G$ and $\Omega$ and letting
	\[
		\frac{b+2i}{b-2i} e^{2w} = z, 
	\]
so that
	\[
		G= z, \qquad
		\Omega = \lambda \frac{\dif{z}}{z^2}, \qquad
		\lambda = \frac{4+b^2}{16},
	\]
we can use Fact~\ref{fact:sol1} to compute the explicit parametrizations for hyperbolic catenoid:
	\[
		X^{\mathrm{H}} 
			= \begin{pmatrix}
				e^{2u + bv} & e^{(b + 2i)v} \\
				e^{(b-2i)v} & e^{-2u + bv}
			\end{pmatrix}
			= R^\mathrm{H}(u)
			e^{bv}\begin{pmatrix}
				1 & e^{2iv} \\
				e^{-2iv} & 1
			\end{pmatrix}
			R^\mathrm{H}(u)^\star
	\]
where $w = u + i v$ (see Figure~\ref{Hcatenoid}).

\begin{figure}
	\centering
	\includegraphics[width=0.4\textwidth]{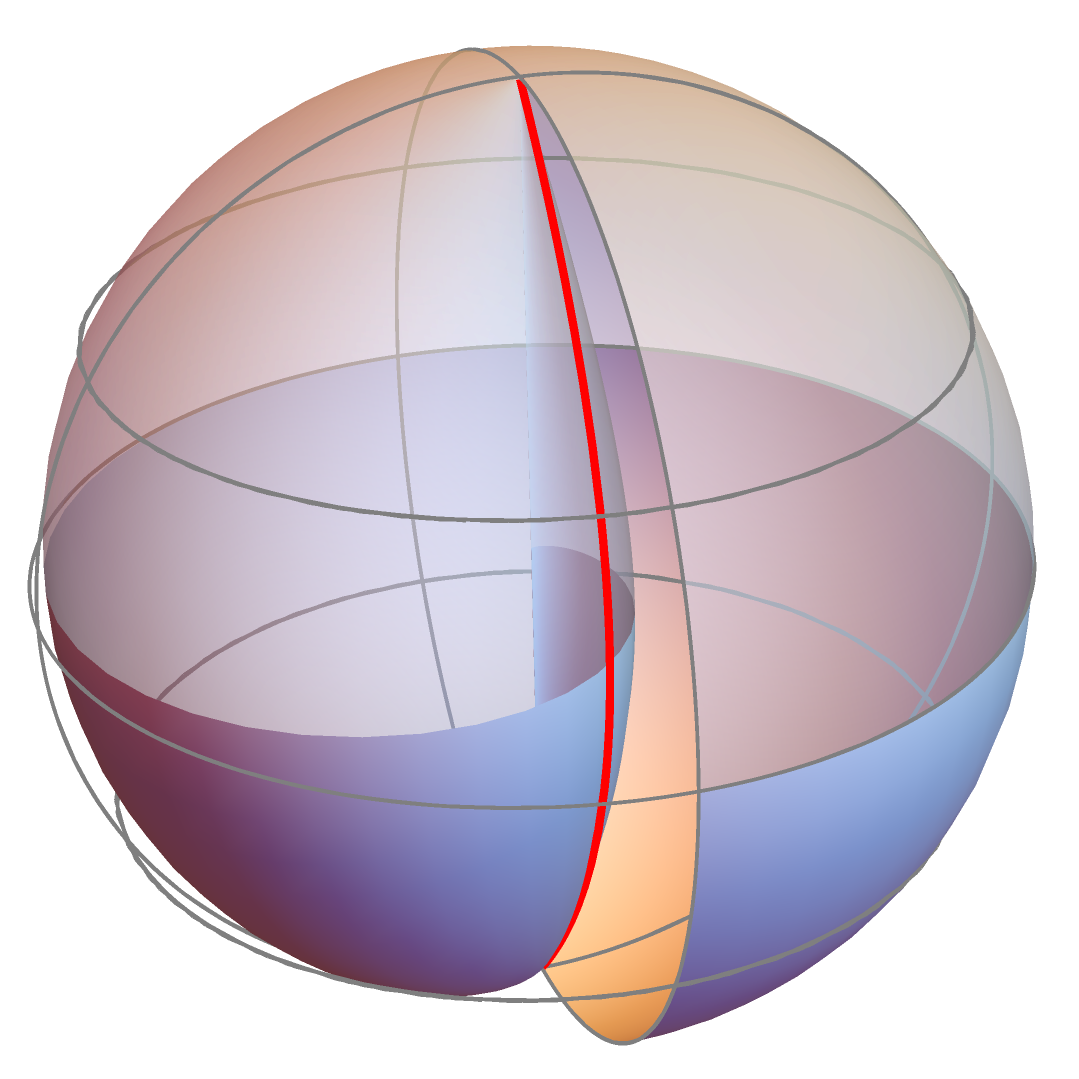}
	\caption{Hyperbolic catenoid in the $3$-dimensional light cone with parameter $b = \frac{3}{2}$, where the initial given curve is highlighted.}
	\label{Hcatenoid}    
\end{figure}



\subsection{Parabolic catenoids}
Finally, we consider the parabolic rotation
	\[
		R^\mathrm{P}(u) = \begin{pmatrix} 1 & u-1 \\ 0 & 1 \end{pmatrix}
	\]
applied to the point $\begin{psmallmatrix} 1 & 1 \\ 1 & 1 \end{psmallmatrix}$ to obtain a parabolic circle
	\[
		\gamma(u) = R^\mathrm{P}(u)
			\begin{pmatrix} 1 & 1 \\ 1 & 1 \end{pmatrix}
			R^\mathrm{P}(u)^\star
			= \begin{pmatrix} u^2 & u \\ u & 1 \end{pmatrix}.
	\]
Note that $\langle \dot{\gamma}, \dot{\gamma} \rangle = 1$ and that $\gamma(1) = \begin{psmallmatrix} 1 & 0 \\ 1 & 1 \end{psmallmatrix} f_3 \begin{psmallmatrix} 1 & 0 \\ 1 & 1 \end{psmallmatrix}^\star$.

We note that any $\mathcal{L}$ satisfying the conformality condition must satisfy
	\[
		\mathcal{L}_\pm(u) := f(u) \gamma(u) \pm f_2
	\]
for some function $f$, and direct calculation tells us that $\mathcal{L}_-$ fails the orientability condition.
Thus, we choose $\mathcal{L} = \mathcal{L}_+$, and the fact that $\mathcal{L}$ must generated by the rotation under consideration allows us to deduce that
	\[
		\mathcal{L}(u) =  c \gamma(u) + f_2.
	\]
for some constant $c$.

Using $\gamma$ and $\mathcal{L}$ as the Bj\"{o}rling data, we calculate that
	\[
		G = u + \frac{2i}{c}, \quad\text{and}\quad \Omega = \frac{c^2}{4} \dif{u}.
	\]
Analytically extending $G$ and $\Omega$ and making change of coordinates so that
	\[
		w + \frac{2i}{c} = z, 
	\]
we have
	\[
		G= z, \qquad
		\Omega = \frac{c^2}{4} \dif{z}.
	\]
This allows us to use Fact~\ref{fact:sol2} to obtain explicit parametrizations for the parabolic catenoids:
	\[
		X^{\mathrm{P}} 
			= \begin{pmatrix}
				e^{cv}(u^2 + v^2) & e^{cv}(u +i v) \\
				e^{cv}(u-i v) & e^{cv}
			\end{pmatrix}
			= R^\mathrm{P}(u + 1)
			e^{cv} \begin{pmatrix}
				v^2 & i v \\
				-iv & 1
			\end{pmatrix}
			R^\mathrm{P}(u + 1)^\star
	\]
where $w = u + iv$ (see Figure~\ref{Pcatenoid}).

\begin{figure}
	\centering
	\includegraphics[width=0.4\textwidth]{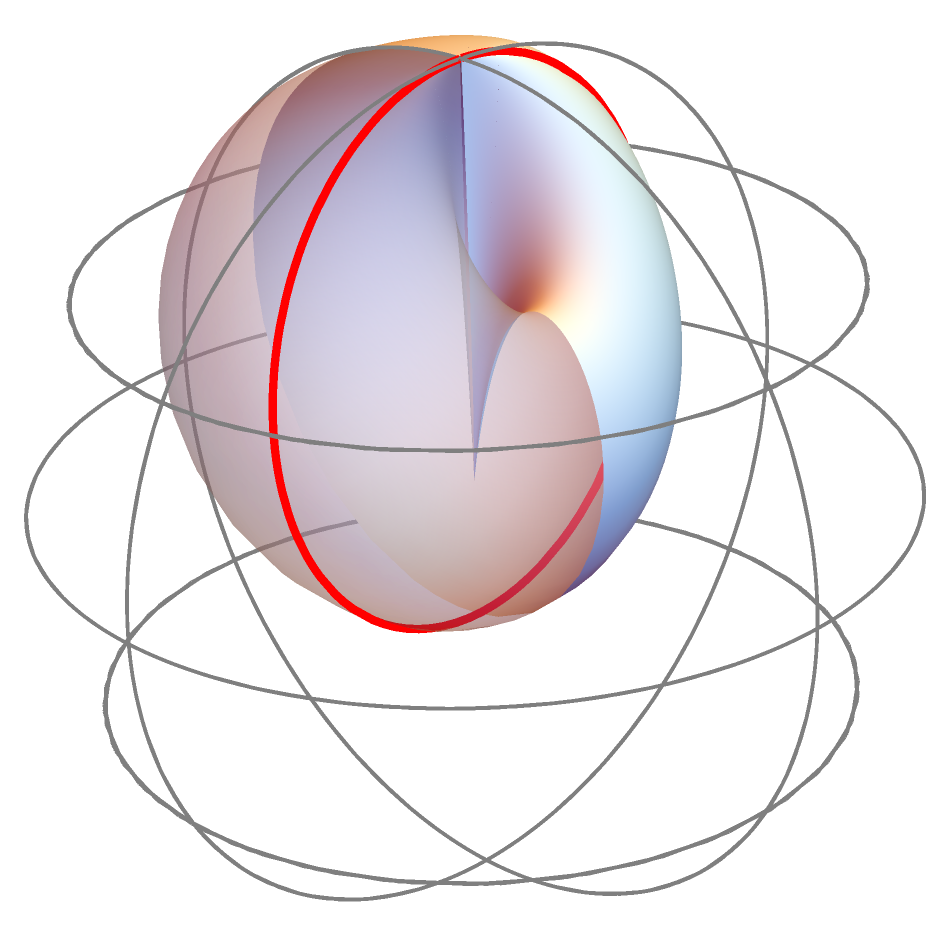}
	\caption{Parabolic catenoid in $3$-dimensional light cone with parameter $b = \frac{1}{2}$, where the initial given curve is highlighted.}
	\label{Pcatenoid}
\end{figure}


\subsection{Classification of rotationally invariant zero mean curvature surfaces}
Any circle in $\qthreep$ is congruent to one of the circles we constructed as orbits of points under rotations up to homotheties and isometries.
Thus we conclude:
\begin{theorem}\label{thm:class}
	Any rotationally invariant zero mean curvature surfaces in $\qthreep$ must be a piece of one of the following surfaces (given with its respective Weierstrass data):
	\begin{itemize}
		\item elliptic catenoid $(G = \frac{a-2}{a+2}e^{2iw}, \Omega = \frac{-i(a+2)^2}{8}e^{-2iw} \dif{w})$
		\item hyperbolic catenoid $(G = \frac{b+2i}{b-2i}e^{2w}, \Omega = \frac{(b-2i)^2}{8}e^{-2w} \dif{w})$, or
		\item parabolic catenoid $(G = w + \frac{2i}{c}, \Omega=\frac{c^2}{4} \dif{w})$
	\end{itemize}
up to homotheties and isometries of $\qthreep$.
\end{theorem}

\subsection{Additional example with analytic extensions}\label{sec:analytic}
As in the parabolic catenoid case, let us take the parabolic circle as the initial curve, i.e.\
	\[
		\gamma(u) = R^\mathrm{P}(u)
			\begin{pmatrix} 1 & 1 \\ 1 & 1 \end{pmatrix}
				R^\mathrm{P}(u)^\star
			= \begin{pmatrix} u^2 & u \\ u & 1 \end{pmatrix}.
	\]
We have seen that $\mathcal{L} = f \gamma + f_2$ for any function $f$ satisfies both the conformality condition and the orientability condition.

Now, if we take
	\[
		f(u) = \frac{c}{u}
	\]
for $u \in (0, \infty)$, then we have $\mathcal{L} = \frac{c}{u} \gamma + f_2$ is not generated by the rotation under consideration; thus, the resulting surface will not be rotationally invariant.
We can still calculate that
	\[
		G = \frac{c + 2i}{c} u, \qquad \Omega = \frac{c^2}{4u^2} \dif{u}.
	\]
Analytically extending $G$ and $\Omega$ and letting
	\[
		\frac{c + 2i}{c} w =: z, 
	\]
so that
	\[
		G= z, \qquad
		\Omega = \lambda \frac{\dif{z}}{z^2}, \qquad
		\lambda = \frac{c(c+2i)}{4},
	\]
we can use Fact~\ref{fact:sol1} to compute the explicit parametrizations:
	\[
		X  = e^{cv}\begin{pmatrix}
			e^{2u} & e^{u+iv} \\
			e^{u -i v} & 1
			\end{pmatrix},
	\]
where $w = e^{u+iv}$ (see Figure~\ref{fig:paraTwosurf}).

\begin{figure}
	\centering
	\begin{minipage}{0.46\linewidth}
		\centering
		\includegraphics[width=0.8\textwidth]{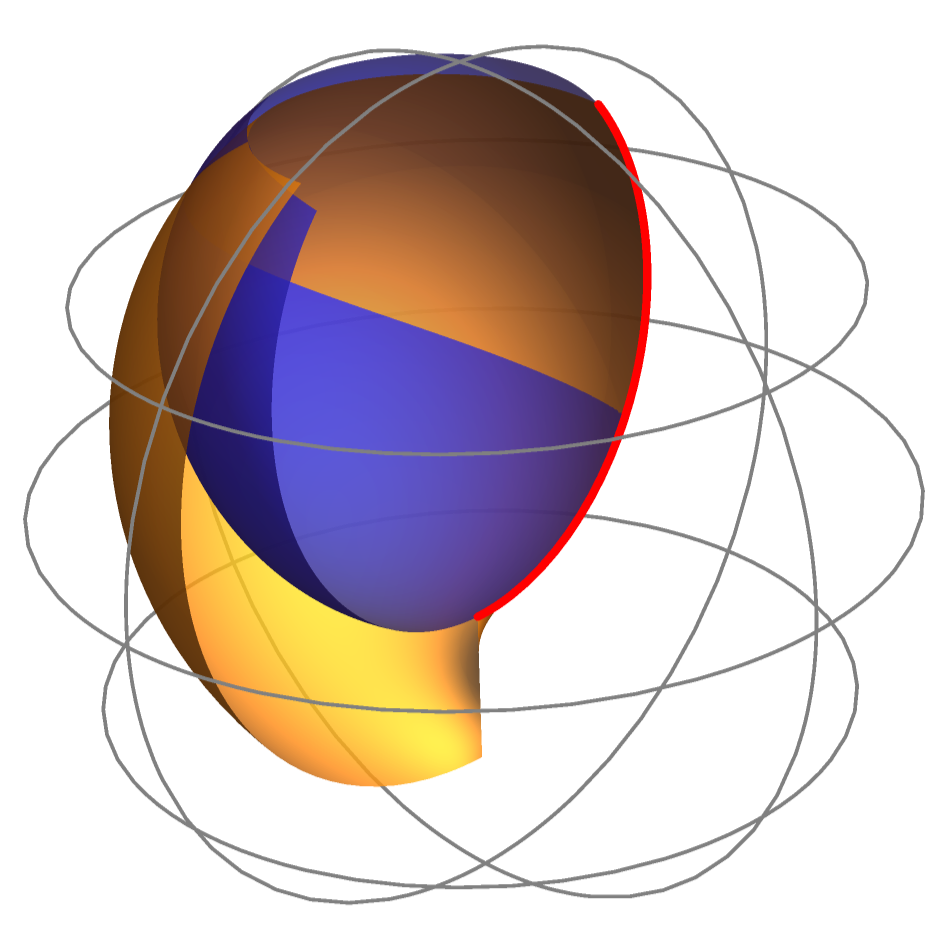}
	\end{minipage}
	\begin{minipage}{0.46\linewidth}
		\centering
		\includegraphics[width=0.8\textwidth]{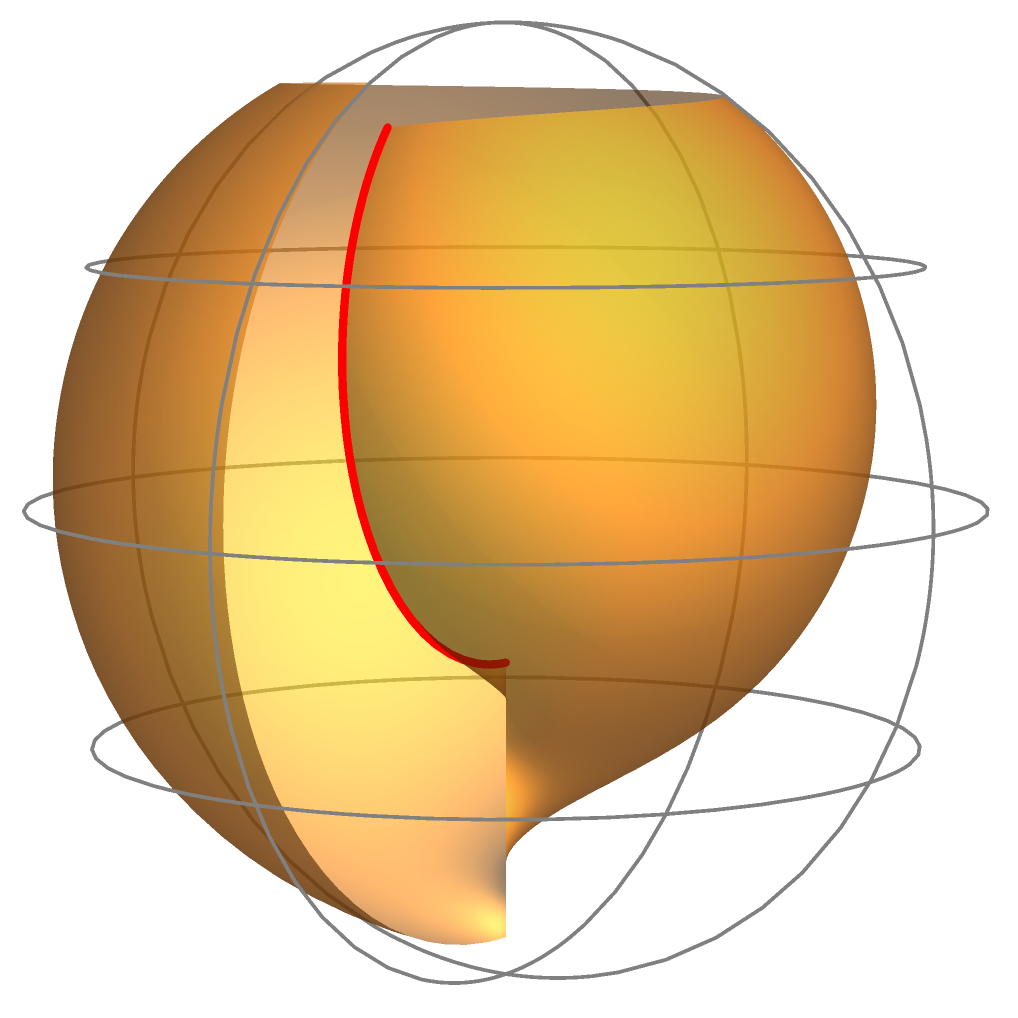}
	\end{minipage}
	\caption{A parabolic catenoid and a non-rotational zero mean curvature surface sharing the same initial curve, where the  initial given curve is highlighted (on the left); the non-rotational zero mean curvature surface drawn over bigger domain (on the right).}
	\label{fig:paraTwosurf}  
\end{figure}

If we change parameters so that $\tilde{u} = e^u$, then we obtain 
		\[
			\tilde{X} = e^{cv} \begin{pmatrix}
				\tilde{u}^2 & e^{iv} \tilde{u} \\
				e^{-iv} \tilde{u} & 1
				\end{pmatrix}
		\]
	so that
		\[
			\langle \tilde{X}_v, \tilde{X}_v \rangle = e^{2bv} \tilde{u}^2.
		\]
	This implies $L(v) := \tilde{X} (0, v)$ is a lightlike curve in $\qthreep$.
	In fact,
		\begin{align*}
			L(v) &= \tilde{X}(0,v) = \begin{pmatrix} 0 & 0 \\ 0 & e^{cv} \end{pmatrix}\\
				&= \begin{pmatrix} e^{\frac{-cv}{2}} & 0 \\ 0 & e^{\frac{cv}{2}} \end{pmatrix}
					\begin{pmatrix} 0 & 0 \\ 0 & 1 \end{pmatrix}
					\begin{pmatrix} e^{\frac{-cv}{2}} & 0 \\ 0 & e^{\frac{cv}{2}} \end{pmatrix}^\star \\
				&= R^\mathrm{H} (\tfrac{-cv}{2}) \begin{pmatrix} 0 & 0 \\ 0 & 1 \end{pmatrix} R^\mathrm{H} (\tfrac{-cv}{2})^\star,
		\end{align*}
	and thus $L$ is a lightlike circle.
	Therefore, we have that $\tilde{X}$ is an analytic extension of $X$ across the lightlike circle $L$  (see Figure~\ref{fig:analExt}).
	
\begin{figure}
	\centering
	\begin{minipage}{\linewidth}
		\centering
		\includegraphics[width=0.4\textwidth]{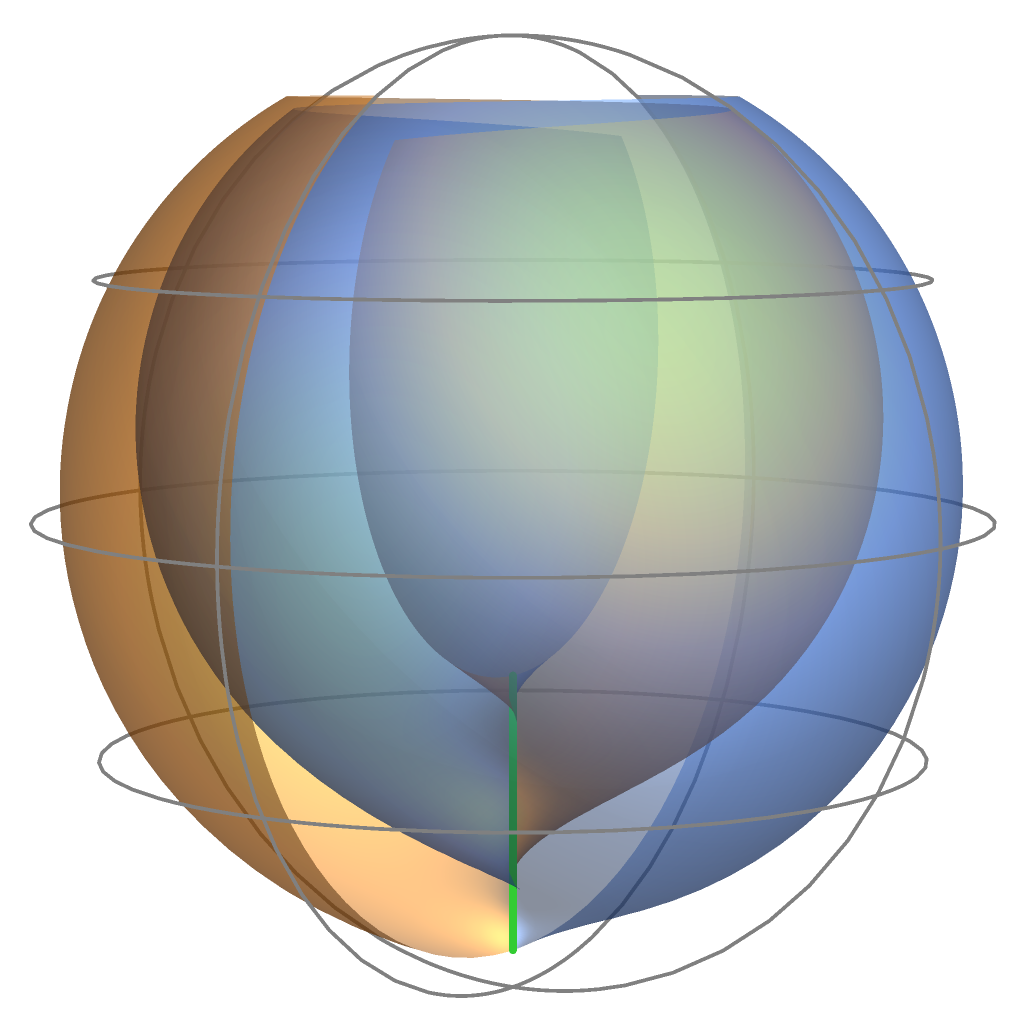}
	\end{minipage}
	\begin{minipage}{0.35\linewidth}
		\centering
		\includegraphics[width=\textwidth]{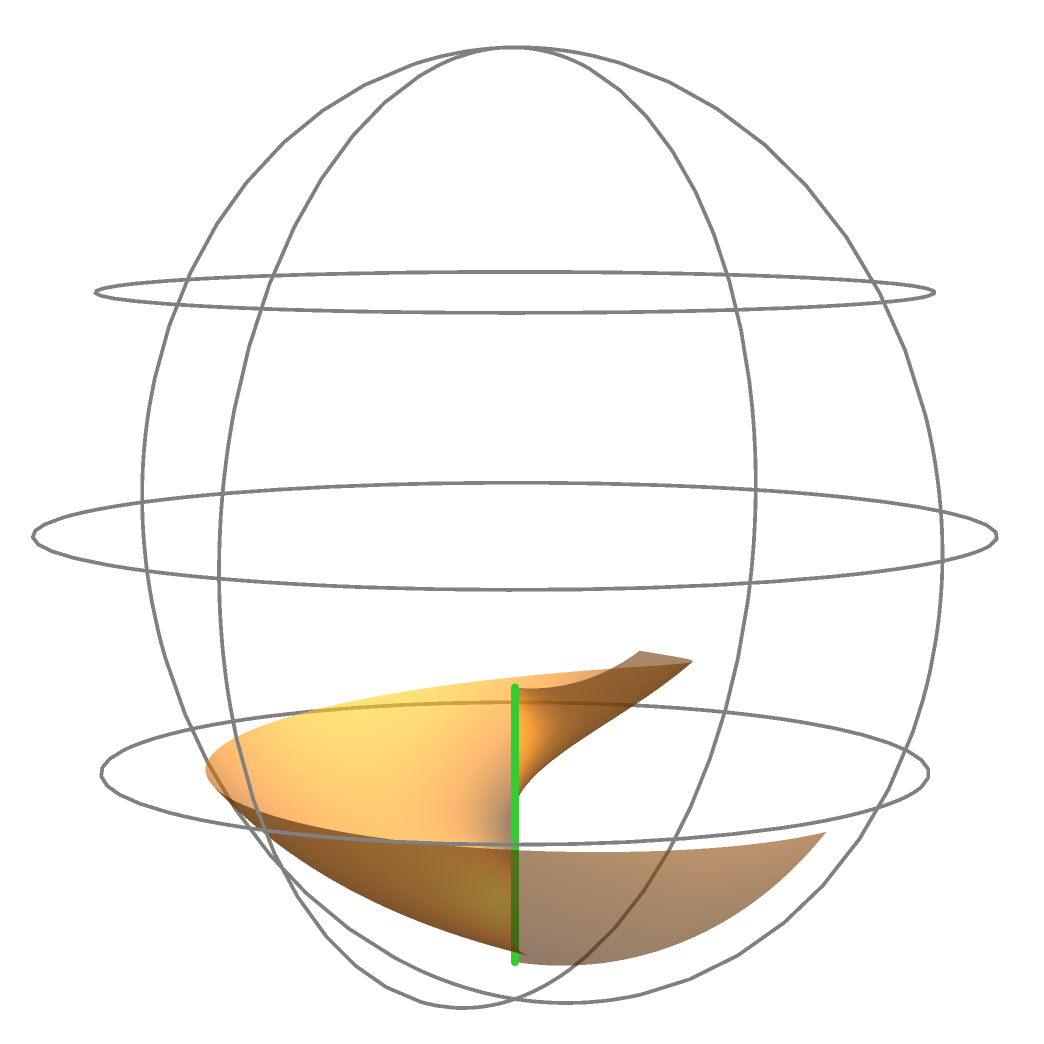}
	\end{minipage}
	\begin{minipage}{0.35\linewidth}
		\centering
		\includegraphics[width=\textwidth]{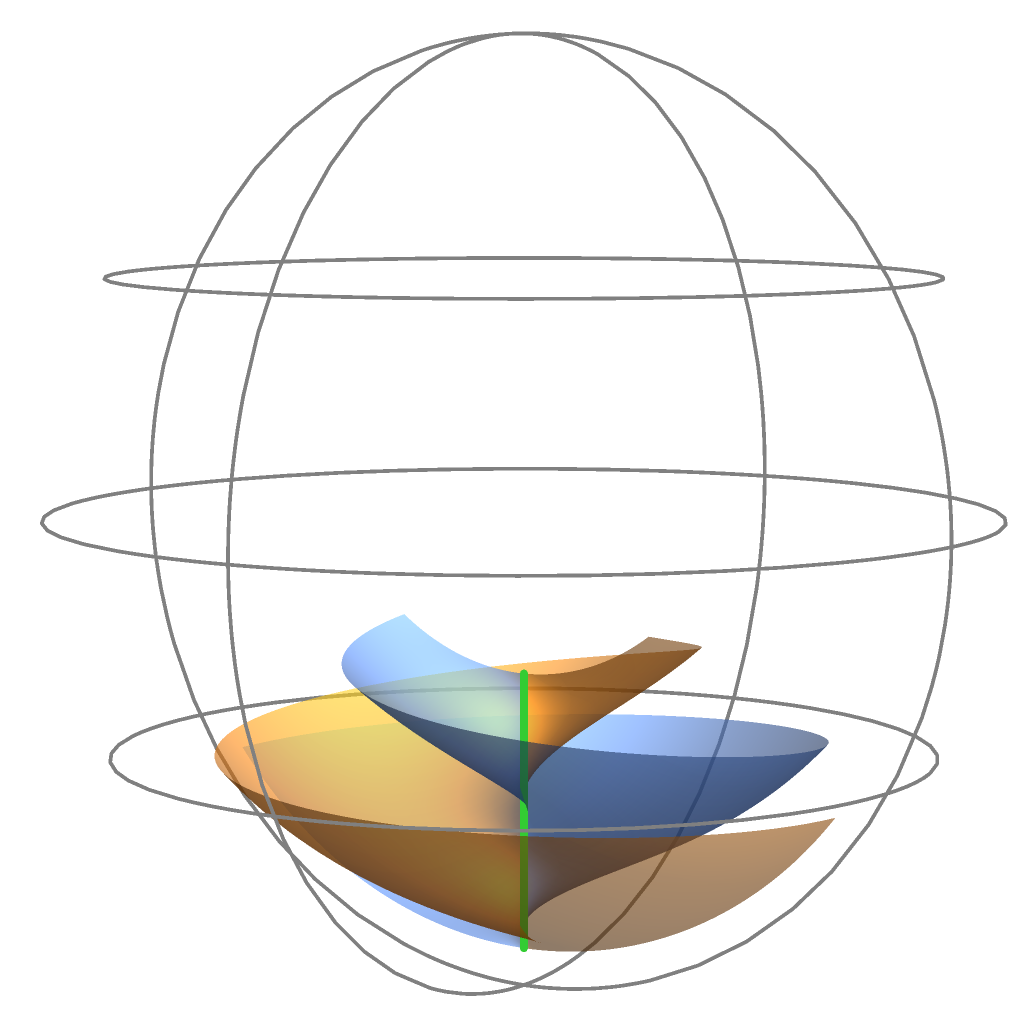}
	\end{minipage}
	\caption{Analytic extension of the non-rotational zero mean curvature surface across a lightlike circle (on the top); closer look at the analytic extension near the lightlike circle (on the bottom). On all figures, the lightlike circle is highlighted.}
	\label{fig:analExt}  
\end{figure}

\textbf{Acknowledgements.}
The last author gratefully acknowledges the support from NRF of Korea (2017R1E1A1A03070929 and NRF 2020R1F1A1A01074585).

\end{document}